\documentclass[12pt,a4paper]{amsart}

\usepackage{amssymb,amscd,latexsym,amsxtra,mathtools}
\usepackage[usenames, dvipsnames]{color}
\usepackage[scr=boondoxo,scrscaled=1.05]{mathalfa}
\usepackage{enumerate}
\usepackage{leftidx}
\usepackage{dsfont}
\usepackage[all]{xy}
\usepackage{xfrac}
\usepackage{units}

\newtheorem{thm}{Theorem}[section]
\newtheorem{cor}[thm]{Corollary}

\newtheorem{prop}[thm]{Proposition}
\newtheorem{lemma}[thm]{Lemma}
\newtheorem{sublemma}[thm]{Sublemma}

\theoremstyle{definition}
\newtheorem{defn}[thm]{Definition}

\newtheorem{prob}[thm]{Problem}

\theoremstyle{remark}
\newtheorem{rem}[thm]{Remark}
\newtheorem{obs}[thm]{Observation}

\def\E{{\mathscr E}}
\def\F{{\mathscr F}}
\def\G{{\mathscr G}}

\def\A{{\mathscr A}}
\def\B{{\mathscr B}}

\def\U{{\mathscr U}}

\def\C{{\mathscr C}}
\def\P{{\mathscr P}}

\def\iff{{\Longleftrightarrow}}

\begin{document}

\title{Square compactness and the Filter Extension Property}

\author{David Buhagiar} 
\address{David Buhagiar \\ Department of
Mathematics \\ Faculty of Science \\ University of Malta \\ Msida
MSD2080, Malta} 
\email{david.buhagiar@um.edu.mt} 

\author{Mirna D\v{z}amonja}\address{Mirna D\v{z}amonja \\ School of
Mathematics \\ University of East Anglia \\ Norwich NR4 7TJ, UK}
\email {M.Dzamonja@uea.ac.uk}

\date{\today}

\thanks{Mirna D\v{z}amonja thanks the University of Malta for their
hospitality during summers of 2017, 2018 and 2019, when this work was
done, and the Isaac Newton Institute in Cambridge, UK and Simons
Foundation for their support during the {\em Mathematical,
Foundational and Computational Aspects of the Higher Infinite (HIF)
programme}, where she learned about the square compactness from Istvan
Juh{\' a}sz. Finally, we thank the anonymous referee for the careful and prompt referee report.
}

\subjclass[2010]{Primary 03E55, 54D20; Secondary 54F05}
\keywords{square compact cardinal, box topology, filter extension property.}

\maketitle

\begin{abstract} We show that the consistency strength of $\kappa$
being $2^\kappa$-square compact is at least weak compact and strictly
less than indescribable. This is the first known improvement to the
upper bound of strong compactness obtained in 1973 by Hajnal and
Juh{\' a}sz. \end{abstract}

\section{Introduction}
In their 1973 paper \cite{HJ}, Andr{\' a}s Hajnal and Istv{\' a}n
Juh{\' a}sz introduced a large cardinal property based on the
productivity of $\kappa$-compactness, the square compact cardinals.
Following and expanding on \cite{HJ} we define that notion:

\begin{defn} An infinite cardinal $\kappa$ is said to be {\em
$\lambda$-square compact} if for any $(<\kappa)$-compact space
$X$\footnote{We note that for simplicity we only work with Hausdorff
spaces.}, with $w(X)\leq\lambda$, the square $X^2 = X\times X$ is
also $(<\kappa)$-compact.

$\kappa$ is said to be \emph{square compact} if it is $\lambda$-square
compact for all $\lambda$ and {\em weak square compact} if it is
$\kappa$-square compact.
\end{defn}

It is not difficult to see that if $\kappa$ is $\lambda$-square
compact, and $X,Y$ are $(<\kappa)$-compact spaces of weight
$\leq\lambda$, then $X\times Y$ is also $(<\kappa)$-compact. Indeed,
first note that in these circumstances the space $X\oplus Y$ is
$(<\kappa)$-compact with weight $\leq\lambda$.  Consequently,
$(X\oplus Y)\times (X\oplus Y)$ is also $(<\kappa)$-compact whenever
$\kappa$ is $\lambda$-square compact. But $X\times Y$ is a closed
subspace of $(X\oplus Y)\times (X\oplus Y)$, and therefore is
$(<\kappa)$-compact.

In \cite{HJ} the authors prove that a square compact cardinal is
weakly compact. It is easy to see and was certainly known to the
authors of \cite{HJ} that a $\lambda$-strongly compact cardinal is
$\lambda$-square compact. The exact consistency strength of the notion
of $\lambda$- square compactness and square compactness, in particular
if the latter is equivalent to weak compactness has been an open
problem since \cite{HJ}. In this paper we obtain a much better upper
bound for the consistency strength of $\lambda$-square compact
cardinals $\kappa$ for $\lambda\le 2^\kappa$, showing that it is
consistent modulo large cardinals that the first weakly compact
cardinal $\kappa$ is $2^\kappa$-square compact (Theorem
\ref{main}(2)). In particular, for such $\lambda$ there cannot be any
ZFC implication from being $\lambda$-square compact to being
indescribable, measurable (or any other large cardinal property that cannot be
exhibited by the first weakly compact cardinal), yet alone to being
$\lambda$-strongly compact. We also show that inaccessible
$\kappa$-square compact cardinals are already weakly compact (Theorem
\ref{main}(1)).

Surprisingly enough, our method does not apply at all to fully square
compact cardinals $\kappa$ or to those that are $\lambda$-square
compact for $\lambda> 2^\kappa$: the best known upper bound for the
consistency strength of them after this work still remains a strongly
compact cardinal.

The most interesting case for our work is when
$\lambda^{<\kappa}=\lambda$. We go through introducing the notion of
$\lambda$-filter extension property for $\kappa$ and showing that it
implies $\kappa$ being $\lambda$-square compact (Theorem
\ref{fromextensiontosquare}) and is equivalent to $\kappa$ being
nearly $\lambda$-strongly compact (Theorem \ref{T:EmbCh}), which is a
large cardinal notion introduced by White in \cite{W}. This allows to
make a connection with Schanker's notion of nearly
$\lambda$-supercompactness, whose consistency strength was studied in
\cite{Codyetal} and whose results we use to obtain our upper bound.

A side gain of our work are numerous equivalences obtained for the
notion of $\lambda$-filter extension property, including topological
ones (Theorem \ref{T:LFEP2}) and the elementary embedding ones
(Theorem \ref{fromextensiontosquare}), which then lead to various
equivalences of weak compactness (Theorem \ref{T:ECwCC}). The
topological equivalences use the $(<\kappa)$-box topology and lead to
a better understanding of it. For example, $\kappa$ has the
$\lambda$-filter extension property if and only if $2^\lambda$ is
$(<\kappa)$-compact in the $(<\kappa)$-box topology (Theorem
\ref{T:LFEP2}).

\section{Notation, conventions and observations}

Throughout our work we shall keep the convention that $2=\{0,1\}$ and
that for any set $X$, any subset of $X$ is identified with its
characteristic function. Furthermore the set of all functions from $X$
to $2$ identified with the product $2^X$, so we have a triple
identification between $\P(X)$, ${}^X 2$ and $2^X$. 

Greek letters $\kappa$  and $\lambda$ will always stand for infinite cardinal numbers.

Recall that a filter $\F$ on a set $X$ is said to be {\em principal}
if it is of the form $\{A \subseteq X:\, Y\subseteq A\}$ for some
$\emptyset\neq Y\subseteq X$, in which case $Y=\bigcap \F$. Otherwise, it is {\em
non-principal} (and hence $\bigcap\F=\emptyset$). A filter is said to
be a \emph{proper} filter if $\F\neq\P(X)$. For a regular cardinal
$\kappa$, we say that a filter $\F$ is {\em $(<\kappa)$-complete} if
$\bigcap\A\in\F$ for every family $\A\subseteq \F$ of size less than
$\kappa$.  It is not difficult to see that if $\F$
is a non-principal $(<\kappa)$-complete filter on some set $X$, then
every set in $\F$ has size at least $\kappa$.

Let $X$ be an infinite set and $\F_0$ a non-principal $(<\kappa)$-complete
filter on $X$. In some of our proofs the objective will be to obtain a
$(<\kappa)$-complete filter $\F$ on $X$ extending $\F_0$ and having
some special property. We give the following definition with this
purpose in mind.

\begin{defn}\label{grill}
Let $\F$ be a filter on a set $X$. Consider the set $G(\F)$, the
\emph{grill associated with} $\F$, of all subsets $S$ of $X$ such that
$S$ meets every $F\in\F$. One observes
that
\[
  S\in G(\F) \setminus \F\quad \text{ if and only if }\quad X\setminus
  S \in G(\F)\setminus \F\,.
\]
Therefore, $\F = G(\F)$ if and only if $\F$ is an
ultrafilter and
\[
  G(\F) = \{S\subseteq X: X\setminus S\notin \F\}\,.
\]
\end{defn}

Note that $\F\subseteq G(\F)$ and that $G(\F)$ satisfies the following
conditions:
\begin{enumerate}
\item $\emptyset \notin G(\F)$.
\item If $S\in G(\F)$ and $S\subseteq T \subseteq X$, then $T\in
G(\F)$.
\item If $\F$ is $(<\kappa)$-complete, then for every collection
$\{S_\lambda:\lambda\in\Lambda\}$ of subsets of $X$, with $|\Lambda| <
\kappa$, if $\bigcup_{\lambda\in\Lambda}S_\lambda \in G(\F)$, at least
one $S_\lambda$ is in $G(\F)$.
\end{enumerate}

The following is easily verified.

\begin{prop}\label{L0} Let $\F_0$ be a proper filter on $X$ and $S\subseteq X$
such that neither $S$ nor $X\setminus S$ is in $\F_0$. Then
\[
\{F\cap L: F\in\F_0 \text{ and }S\subseteq L\subseteq X\}
\]
generates a proper filter $\F$ that contains
$\{S\}\cup\F_0$. Moreover, if $\F_0$ is $(<\kappa)$-complete, then so
is $\F$.
\end{prop}

\begin{defn} Let $X$ be a set. A subset $\A$ of $\P(X)$
\emph{measures} a subset $A$ of $X$ if $A\in \A$ or $X\setminus A\in \A$. 
$\A$ \emph{measures} a subset $\C$ of $\P (X)$ if $\A$ measures
every element $C$ of $\C$.

$\kappa$ is said to have the \emph{filter property} if for every
subset $\C$ of $\P(\kappa)$ of size $\leq\kappa$, there is a
$(<\kappa)$-complete non-principal filter $\F$ on $\kappa$ which
measures $\C$.
\end{defn}

So $\F$ is an ultrafilter on $\kappa$ if it measures $\P(\kappa)$. It is
well-known that an uncountable cardinal $\kappa$ has the filter
property if and only if $\kappa$ is weakly compact.

For cardinals $\kappa\le\lambda$ we write $\P_\kappa(\lambda)$ for the family
$\{a\subseteq \lambda:\,|a|<\kappa\}$.

\begin{defn}
A filter $\F$ on $\P_\kappa(\lambda)$ is said to be
\emph{fine} if for every $\alpha<\lambda$ we have $B_\alpha = \{a \in
\P_\kappa(\lambda):\alpha\in a\}\in \F$. 
\end{defn}

Note that if $\F$ is a fine
($<\kappa$)-complete filter on $\P_\kappa(\lambda)$, then $C_a = 
\{b \in \P_{\kappa}(\lambda):\,a\subseteq b\}\in\F$ for every
$a\in \P_{\kappa}(\lambda)$. Indeed, $C_a = \bigcap_{\alpha \in
  a}B_\alpha$ and $|a|<\kappa$.

We now mention some notions from topology.

It is well known that even for $T_0$ spaces $|X|\leq 2^{w(X)}$, which is
a fact that we shall use without further mention.

A topological notion connected to the $(<\kappa)$-completeness of filters is that
of $(<\kappa)$-compactness.

\begin{defn} A topological space $X$ is said to be {\em
$(<\kappa)$-compact}, if every open cover of $X$ has a subcover of
cardinality strictly less than $\kappa$.
\end{defn}

This notion is sometimes easier to express in terms of the
complementary property of closed sets. 

\begin{defn} Let $X$ be a set. A subset $\A$ of $\P(X)$ is said to
satisfy the $(<\kappa)$-intersection property if $\bigcap_{i<\gamma} A_i
\neq \emptyset$ whenever $\gamma<\kappa$ and
$\{A_i:i<\gamma\}\subseteq \A$.
\end{defn}

The following observation then connects the notions of compactness and completeness.

\begin{obs}\label{T:III.5} Suppose that $\kappa$ is a regular
cardinal.  A topological space $X$ is $(<\kappa)$-compact if and only
if every $(<\kappa)$-complete filter consisting of closed sets in $X$
has a nonempty intersection if and only if every family of closed sets
with the $(<\kappa)$-intersection property, has a nonempty
intersection.
\end{obs}

Filters are not only connected to compactness but also to
the productiveness of that notion. Let us recall some
product topology definitions.

\begin{defn} Suppose that $\langle X_\alpha:\,\alpha
<\alpha^\ast\rangle$ is a sequence of topological spaces and $\theta$
is an infinite cardinal. The $(<\theta)${\em -box topology} on the
product $\Pi_{\alpha<\alpha^\ast}X_\alpha$ has as a base, sets of
the form $\Pi_{\alpha<\alpha^\ast}O_\alpha$, where each $O_\alpha$ is
open in $X_\alpha$ and $|\{\alpha <\alpha^\ast:\, O_\alpha\neq
X_\alpha\}|< \theta$.
\end{defn}

So, in the case of $\theta=\aleph_0$, the $(<\theta)$-box topology is
the usual Tychonoff topology.  In the case of $\kappa=\aleph_0$,
Observation \ref{T:III.5} has another equivalent, which is that $X$ is
compact if and only if every ultrafilter of closed subsets of $X$ has
a nonempty intersection. The forward direction of this claim is
clear, while for the backward direction, suppose that we have a filter
$\F$ consisting of closed subsets of $\kappa$. Consider the collection
$\C$ of all families of closed subsets of $X$ which contain $\F$ and
which form a filter. Since $\C$ is nonempty, partially ordered by
$\subseteq$ and closed under unions, by Zorn's lemma $\C$ has a
maximal element $\U$, which is then an ultrafilter of closed
sets. Hence $\bigcap \U\neq \emptyset$ and so $\bigcap \F\neq
\emptyset$. However, if $\kappa >\aleph_0$, the notion of
$(<\kappa)$-compactness does not have an ultrafilter
characterisation. The above proof does not work since Zorn's lemma
does not apply to the family of all $(<\kappa)$-complete filters of
closed sets to produce a $(<\kappa)$-complete ultrafilter, but actually
that analogue of Zorn's lemma leads to large cardinal notions, as we
shall discuss below.

The ultrafilter equivalent of compactness is used in Tychonoff's proof
that the Tychonoff product of compact spaces is compact and, for the
reasons just explained, the analogue for the $(<\kappa)$-box
topological product of $(<\kappa)$-compact spaces does not work. In
fact, here is a well-known result for {\em strongly compact}
cardinals.

\begin{defn}\label{def:stronglycompact}
A cardinal $\kappa$ is said to be {\em
$\lambda$-strongly compact} if every ($<\kappa$)-complete filter $\F$
over any set $S$, generated by at most $\lambda$ many elements, can be
extended to a ($<\kappa$)-complete ultrafilter over $S$.

$\kappa$ is {\em strongly compact} if it is $\lambda$-strongly compact
for all $\lambda\ge\kappa$.
\end{defn}

\begin{thm}\label{T:SCEq} The following statements are equivalent:
\begin{enumerate}[{\rm (1)}]
\item The cardinal $\kappa$ is strongly compact;
\item Any $(<\kappa)$-box product of $(<\kappa)$-compact spaces is
$(<\kappa)$-compact;
\item Any $(<\kappa)$-box product of discrete two-point spaces is
$(<\kappa)$-compact.
\end{enumerate}
\end{thm}

Although this theorem is attributed to the mathematical folklore, we
could not find a proof written in the literature. One of our main
results, Theorem \ref{T:ECwCC} below, is an elaboration of a similar
characterisation but for weakly compact cardinals. A proof of Theorem
\ref{T:SCEq} can be deduced by similar arguments as those that we used
for the proof of Theorem \ref{T:ECwCC}.

\section{$\lambda$-Filter Extension Property}\label{S:LFEP}
Let us start with the following definition.

\begin{defn}\label{D:2} Suppose that $\kappa$ is an infinite
cardinal and $\lambda\geq\kappa$. $\kappa$ is said to have the \emph{$\lambda$-filter
extension property} if for every family
$\A\subseteq \P(2^\lambda)$ with $|\A|\leq\lambda$, every
($<\kappa$)-complete filter $\F$ over $2^\lambda$ generated by at most
$\lambda$ many elements can be extended to a ($<\kappa$)-complete
filter over $2^\lambda$ measuring $\A$.
\end{defn}

This section is devoted to the following Theorem \ref{T:LFEP2}, which
will lead to an upper bound for the consistency strength of
$\lambda$-square compact cardinals, as per Theorem
\ref{fromextensiontosquare}.

\begin{thm}\label{T:LFEP2} Suppose that $\lambda^{<\kappa}=\lambda$
for some infinite cardinals $\kappa$ and $\lambda$. Then the following
are equivalent:
\begin{enumerate}[{\rm (1)}]
\item $\kappa$ has the $\lambda$-filter extension property;
\item Suppose that $\{ X_\gamma: \,\gamma \le \Gamma\}$ for some
$\Gamma\le \lambda$ is a family of $(<\kappa)$-compact spaces such
that $w(X_\gamma) \leq \lambda$ for all $\gamma\in\Gamma$.  Then the
$(<\kappa)$-box product space $\prod_{\gamma \in \Gamma}X_\gamma$ is
also $(<\kappa)$-compact;
\item $2^{\lambda}$ is $(<\kappa)$-compact in the $(<\kappa)$-box
topology.
\end{enumerate}
\end{thm}

To prove Theorem \ref{T:LFEP2}, we shall prove two lemmas.

\begin{lemma}\label{T:LFEP} Suppose that $\lambda^{<\kappa}=\lambda$
and that $\kappa$ has the $\lambda$-filter extension property. Let
$\Gamma\le \lambda$ and suppose that $\{ X_\gamma: \,\gamma \le
\Gamma\}$ is a family of $(<\kappa)$-compact spaces such that
$w(X_\gamma) \leq \lambda$ for all $\gamma\in\Gamma$.  Then the
$(<\kappa)$-box product space $ \prod_{\gamma \in \Gamma}X_\gamma$ is
also $(<\kappa)$-compact.
\end{lemma}

\begin{proof} Let $X$ denote the $(<\kappa)$-box product space
$\prod_{\gamma \in \Gamma}X_\gamma$. Observe that $w(X) \leq
\lambda$. Indeed, suppose that for each $\gamma$ we have a base
$\B_\gamma$ for the topology on $X_\gamma$ with
$|\B_\gamma|\leq\lambda$. We denote by $\pi_\gamma:X\to X_\gamma$, the
projection of $X$ onto $X_\gamma$. Namely,
\[
\pi_\gamma(x) = x_\gamma \quad \text{ for }x = \langle x_\gamma: \gamma\le\Gamma\rangle.
\]
Then the base $\B$ (for the $(<\kappa)$-box topology) on $X$ is
generated by taking $<\kappa$ intersections from the collection $
\{\pi^{-1}_{\gamma}(U_{\gamma}):U_\gamma\in\B_\gamma\,,
\gamma\in\Gamma\}$ is of cardinality $\leq\lambda$, by the cardinal
arithmetic hypothesis on $\lambda$. It then follows that $X$ has size
$\le 2^\lambda$, so we can view $X$ as a subset of $2^\lambda$.

To show $(<\kappa)$-compactness of $X$, let $\F$ be a closed
collection in $X$ satisfying the $(<\kappa)$-intersection
property. One can also assume that for every $F\in\F$, $X\setminus
F\in\B$.  Since $|\B|\leq\lambda$, it is clear that $|\F| \leq
\lambda$. Note that we can use $\lambda$-filter extension property of
$\kappa$ on $|\F|$ since $X$ is a subset of $2^\lambda$. Therefore,
since $|\B| \leq \lambda$, we can extend $\F$ to a
$(<\kappa)$-complete filter $\F'$ that measures $\B$.  We claim that
for each $\gamma\in\Gamma$
\[
\F_\gamma = \{\pi_\gamma(A): A\in\F'\}
\]
is a $(<\kappa)$-complete filter in $X_\gamma$ that measures
$\B_\gamma$. First it is clear that
$\emptyset\notin\F_\gamma$. Secondly, if $B\supseteq \pi_\gamma(A)$
for some $A\in\F'$, then $\pi^{-1}_\gamma(B)\supseteq A$, which
implies $\pi^{-1}_\gamma(B) \in\F'$. Hence
\[
B = \pi_\gamma(\pi^{-1}_\gamma(B)) \in \F_\gamma\,.
\]
Thirdly, if $B = \pi_\gamma(A), B' = \pi_\gamma(A')$ for some $A, A'
\in \F'$, then
\[
B\cap B' \supseteq \pi_\gamma(A\cap A')\,.
\]
Since $A\cap A'\in \F'$, and hence $\pi_\gamma(A\cap A')\in
\F_\gamma$, this implies (as has just been proved)
\[
B\cap B' \in \F_\gamma\,.
\]
The same is true for any collection $\{B_\alpha:\alpha\in I\}$ of size
$<\kappa$ because $\F'$ is $(<\kappa)$-complete.
Therefore, $\F_\gamma$ is a $(<\kappa)$-complete filter of $X_\gamma$.

In fact $\F_\gamma$ measures $\B_\gamma$, for suppose
$U_\gamma\in\B_\gamma$. Then $\pi^{-1}_\gamma(U_\gamma)\in\B$. Thus,
either $\pi^{-1}_\gamma(U_\gamma)$ or $X\setminus \pi^{-1}_\gamma(U_\gamma)$ is
in $\F'$. In the first case, $U_\gamma\in \F_\gamma$, in the second
case $X_\gamma\setminus U_\gamma\in \F_\gamma$, which is what we
needed to show. Now $\overline{\F_\gamma} = \{C:C\in\F_\gamma, C\text{
is closed}\}$ is a closed collection in $X_\gamma$ satisfying the
$(<\kappa)$-intersection property,
hence its intersection is nonempty. Suppose $x_\gamma\in \bigcap
\overline{\F_\gamma}$, then we can assert that
\[
x = \{x_\gamma: \gamma\in\Gamma\} \in \bigcap\F\,.
\]
To show this, suppose $U\in\B$ is a given basic neighbourhood of $x$. Then $U$ is of
the form
\[
U = \prod_{i\in \Gamma_0} U_{\gamma_i} \times \prod\{X_\gamma: \gamma \neq
\gamma_i, i\in\Gamma_0\}\,,
\]
where each $U_{\gamma_i}\in\B_{\gamma_i}$ is a neighbourhood of
$x_{\gamma_i}$ in $X_{\gamma_i}$, and $\Gamma_0\subseteq\Gamma$ has
cardinality $<\kappa$. Note that $U$ can be expressed as
\begin{equation}\label{Eq:IV.1}
U = \bigcap_{i\in\Gamma_0} \pi^{-1}_{\gamma_i}(U_{\gamma_i})\,.
\end{equation}
Now $\pi^{-1}_{\gamma_i}(U_{\gamma_i})\in\B$. Suppose, however, that
it is not in $\F'$. Then $X\setminus
\pi^{-1}_{\gamma_i}(U_{\gamma_i})\in\F'$, since $\F'$ measures
$\B$. But then $X_{\gamma_i}\setminus U_{\gamma_i}\in\F_{\gamma_i}$, which is
closed. This contradicts $x_{\gamma_i}\in \bigcap
\overline{\F_{\gamma_i}}$. Consequently, we obtain
$\pi^{-1}_{\gamma_i}(U_{\gamma_i})\in\F'$. Hence, in view of
\eqref{Eq:IV.1}, we conclude that $U\in\F'$ because $\F'$ is
$(<\kappa)$-complete. Therefore $x\in\bigcap\F$
and we can conclude that $X$ is a $(<\kappa)$-compact space.
\end{proof}

The next lemma is motivated by the proof for the
equivalence of the Ultrafilter Principle and Tychonoff's Theorem for
compact Hausdorff spaces, see \cite{My, Sch}.

\begin{lemma}\label{T:EqCwCC} Suppose that $\kappa$ is an infinite
cardinal and $\lambda^{<\kappa} = \lambda$.  If $2^{\lambda}$ is
$(<\kappa)$-compact in the $(<\kappa)$-box topology, then $\kappa$ has
the $\lambda$-filter extension property.
\end{lemma}

\begin{proof} If $\kappa=\aleph_0$, the conclusion follows from the
Axiom of Choice.  For $\kappa$ uncountable, we shall need the following observation.

\begin{sublemma}\label{newref} Suppose that $2^\kappa$ is $(<\kappa)$-compact
in the $(<\kappa)$-box topology. Then we have that $2^\theta < \kappa$
whenever $\theta < \kappa$.
\end{sublemma}

\begin{proof} For each $f\in 2^\theta$ define $U_f=\{g\in 2^\kappa:\,
f\subseteq g\}$, so an open set in the $(<\kappa)$-box topology of
$2^\kappa$.  It is not difficult to see $\{ U_f:\, f\in 2^\theta\}$
forms a minimal open cover, so we have that its size satisfies
$2^\theta <\kappa$ by the $(<\kappa)$-compactness of $2^\kappa$.
\end{proof}

Now suppose that $\lambda\geq\kappa>\aleph_0$ satisfy
the hypothesis of the lemma and let $\F_0$ be a $(<\kappa)$-complete
filter on $2^\lambda$ which is generated by a family $\E$ of size
$\leq\lambda$ and let $\A$ be a family of subsets of $2^\lambda$ of size
$\leq\lambda$. We need to produce a $(<\kappa)$-complete filter $\F$
with $\F_0 \subseteq \F$ that measures $\A$. For simplicity in
notation, let us write $X$ for $2^\lambda$. The assumption that $2^{\lambda}$ is
$(<\kappa)$-compact in the $(<\kappa)$-box topology implies that 
$2^\kappa$ is $(<\kappa)$-compact
in the $(<\kappa)$-box topology, since $2^\kappa$ is homeomorphic to
a closed subspace of $2^{\lambda}$. Therefore, by Sublemma \ref{newref},  $2^\theta < \kappa$
whenever $\theta < \kappa$.

Let $\Sigma_0 = \E \cup \A \cup \A^c\cup\{\emptyset\}$, where $\C^c =
\{X\setminus S: S\in \C\}$ for any $\C\subseteq \P(X)$. Furthermore, let
$\Sigma$ be the closure of $\Sigma_0$ under intersections of $<\kappa$ elements.
Finally let $\Sigma'= \Sigma  \cup \Sigma^c$. 

Consider $\P(\Sigma')$,
identified with $2^{\Sigma'}$. Note that $|\Sigma'|\leq \lambda$
because $\lambda^{<\kappa} = \lambda$. Any $\F \in \P(\Sigma')$ is a
collection of subsets of $X$.

Now let $2^{\Sigma'}$ have the $(<\kappa)$-box
product topology; then the $S^{\text{th}}$ coordinate projection
$\pi_S:2^{\Sigma'} \to 2$ for any $S\in\Sigma'$, defined by
\[
\pi_S(\F) = \begin{cases}
               1 &\text{if $S\in\F$,} \\
               0 &\text{if $S\notin\F$,}
              \end{cases}
\]
is continuous (in fact it is even continuous in the Tychonoff product
topology), since for example $\pi_S^{-1}(\{1\})=\{\F:\,S\in \F\}$ is a basic clopen set. 
Let
\begin{align*}
D &= \{\F\in 2^{\Sigma'} : \F \text{ is a proper $(<\kappa)$-complete filter in }\Sigma'\}\,, \\
E &= \{\F\in 2^{\Sigma'} : \F \supseteq \F_0\} \mbox{ and}\\
  \Gamma_S &= \{\F\in 2^{\Sigma'} : S\in\F \text{ or }X\setminus S\in\F\},
             \mbox{ for each $S\in \Sigma'$.} 
\end{align*}
It is not difficult to see that
\begin{align*}
D &= \bigcap_{\substack{\mathllap{\Lambda\subseteq} \Sigma' \\
  \mathllap{|\Lambda|<} \kappa}}\{\F\in 2^{\Sigma'} : ([(\forall A\in\Lambda) \pi_A(\F) = 1 ]
  \text{ iff }\pi_{\bigcap\Lambda}(\F) = 1) \text{ and } \pi_{\emptyset}(\F) = 0\}, \\
E &= \bigcap_{A\in \Sigma'}\{\F\in 2^{\Sigma'} : [1 -
    \pi_A(\F)]\pi_A(\F_0) = 0\}\,, \quad \text{and}\\
\Gamma_S &= \{\F\in 2^{\Sigma'} : \pi_S(\F) +  \pi_{X\setminus S}(\F)
           \geq 1\}\,.
\end{align*}
We claim that by the continuity of the projections, the sets $D, E,
\Gamma_S$ are closed. Let us check this for the set $D$. 

We first note that for any $\Lambda \subseteq \Sigma'$, $|\Lambda| < \kappa$, the set
\begin{equation*}
O_\Lambda = \prod_{S\in \Sigma'} U_S\,, \text{ where } 
U_S = \begin{cases}
               \{0,1\} &\text{if $S\notin\Lambda$,} \\
               \{1\} &\text{if $S\in\Lambda$,}
              \end{cases}
\end{equation*}
is open in the $(<\kappa)$-box product topology. Consequently,
$C_\Lambda = 2^{\Sigma'}\setminus O_\Lambda$ is closed. Now, for any such
$\Lambda$,
\[
\{\F\in 2^{\Sigma'} : ([(\forall A\in\Lambda) \pi_A(\F) = 1 ]
  \text{ iff }\pi_{\bigcap\Lambda}(\F) = 1) \text{ and } \pi_{\emptyset}(\F) = 0\}  
\]
is equal to the set
\[
(C_\Lambda \cap \pi^{-1}_{\bigcap\Lambda}(\{0\}) \cap \pi^{-1}_\emptyset(\{0\})) \cup
((\bigcap_{A\in\Lambda}\pi^{-1}_A(\{1\}))\cap \pi^{-1}_{\bigcap\Lambda}(\{1\})
\cap \pi^{-1}_\emptyset(\{0\}))\,,
\]
which is closed. Thus, taking the intersection over all $\Lambda
\subseteq \Sigma'$, $|\Lambda| < \kappa$, we conclude that $D$ is
closed.

It then follows that the set $\Phi_S = D\cap E\cap \Gamma_S$ is also
closed, for any $S\in\Sigma'$.  Note that $\Phi_S$ is equal to the set
\begin{multline*}
\{\F : \F \text{ is a proper $(<\kappa)$-complete filter in $\Sigma'$ that includes
  $\F_0$} \\ \text{and contains at least one of $S$ or $X \setminus S$}\}\,.
\end{multline*}
By Proposition \ref{L0},  $\Phi_S\neq\emptyset$. 

Recall the notation of grills from Definition \ref{grill} and let us
now show that the collection of nonempty closed sets $\{\Phi_S: S\in
\left(G(\F_0)\setminus \F_0\right)\cap\Sigma'\}$ has the
$(<\kappa)$-intersection property.

Consider a collection $\{\Phi_{S_\alpha}:\alpha\in\Lambda\}$, where
$|\Lambda| < \kappa$. Take any $x\in X$ and for every
$\alpha\in\Lambda$ choose $S^*_\alpha\in\{S_\alpha,X\setminus
S_\alpha\}$ satisfying $x\in S^*_\alpha$. Then $x\in S^* =
\bigcap_{\alpha\in\Lambda}S^*_\alpha$. Thus, by considering all
possible intersections $\bigcap_{\alpha\in\Lambda}S^*_\alpha$ (there
are $2^{|\Lambda|} < \kappa$ many, see the discussion at the beginning of the proof), where
$S^*_\alpha\in\{S_\alpha,X\setminus S_\alpha\}$, we form a partition
of $X$ of cardinality $<\kappa$. Since $X\in G(\F_0)\cap \Sigma'$ it
follows that one of the intersections of the form
$\bigcap_{\alpha\in\Lambda}S^*_\alpha$ must be in $G(\F_0)\cap\Sigma'$
by property (3) of the grills. Let $S' =
\bigcap_{\alpha\in\Lambda}S^*_\alpha\in G(\F_0)\cap\Sigma'$. Note that
$S'\notin \F_0$, otherwise $S^*_\alpha\in\F_0$ and these sets are
taken from $G(\F_0)\setminus\F_0$. Consequently, $S'\in
G(\F_0)\setminus \F_0$ and therefore $X\setminus S'\in
G(\F_0)\setminus \F_0$. It then follows, again by applying Proposition
\ref{L0}, that there exists a $(<\kappa)$-complete filter $\G$ in
$\Sigma'$ that extends $\F_0$ and contains $S'$. Then $\G$ is in the
intersection $\bigcap_{\alpha\in\Lambda}\Phi_{S_\alpha}$, which is
therefore nonempty, as required.

By our assumption, $2^{\Sigma'}$ is $(<\kappa)$-compact; hence there
exists some $\U$ in the intersection $\bigcap_{S\in (G(\F_0)\setminus
  \F_0)\cap\Sigma'}\Phi_S$. Then $\U$ is a $(<\kappa)$-complete filter
on $\Sigma'$ that extends $\F_0$ and measures $\A$. One can then
extend $\U$ to a $(<\kappa)$-complete filter on $\P(X)$ by taking
supersets of sets from $\U$.
\end{proof}

\begin{proof}[Proof of Theorem \ref{T:LFEP2}] The implication
$(1)\implies (2)$ is given by Lemma \ref{T:LFEP}.  The implication
$(2)\implies (3)$ is obvious. The implication $(3)\implies (1)$ is
given by Lemma \ref{T:EqCwCC}.
\end{proof}

Using the same argument as in the proof of Lemma \ref{T:LFEP} we obtain:

\begin{thm}\label{fromextensiontosquare} For any cardinals
$\kappa\le\lambda$, if $\kappa$ has the $\lambda$-filter extension
property, then $\kappa$ is $\lambda$-square compact.
\end{thm}

Of course, a direct consequence of Definitions \ref{def:stronglycompact} and 
\ref{D:2} is

\begin{prop} If $\kappa$ is $\lambda$-strongly compact then it has the
$\lambda$-filter extension property.
\end{prop}

We do not know if the converse of Theorem \ref{fromextensiontosquare}
holds, leading to an open question:

\begin{prob} If a cardinal $\kappa$ is $\lambda$-square compact, does
it necessarily have the $\lambda$-filter extension
property? \end{prob}

\section{Embedding Characterisations of the $\lambda$-filter extension
property}\label{S:EmbChLambda}

Let us recall the following characterisation of $\lambda$-strongly
compact cardinals, 
(\cite[Th. 22.17, Pg. 307]{K}):
\begin{itemize}
\item A cardinal $\kappa$ is $\lambda$-strongly compact if and only if
there exists an elementary embedding $j: V \to M$ with critical point
$\kappa$, such that for all $X\subseteq M$ with $|X|\leq\lambda$,
there exists $Y\in M$ with $X\subseteq Y$ and $M \models |Y| <
j(\kappa)$.
\end{itemize}

We shall prove that the $\lambda$-filter extension property is
equivalent to a ``mini'' version of the above characterisation of
$\lambda$-strongly compact cardinals. That mini version has in fact
appeared in the recent literature, it was introduced by P. A. White in
\cite{W}, under the name {\em nearly $\lambda$-strongly compact
cardinals}.
 
\begin{defn}\label{D:LNSC} The cardinal $\kappa$ is said to be \emph{nearly
$\lambda$-strongly compact} if for every family $\A$ of $\lambda$-many
subsets of $\P_\kappa(\lambda)$ there is a non-principal
($<\kappa$)-complete fine filter $\F$ over $\P_\kappa(\lambda)$
measuring $\A$.
\end{defn}

\begin{thm}\label{T:EmbCh} Suppose that $\kappa>\aleph_0$ and
$\lambda^{<\kappa}=\lambda$. Then:
\begin{itemize}
\item $\kappa$ has the $\lambda$-filter extension property, if and
  only if
\item $\kappa$ is nearly $\lambda$-strongly compact, if and only if
\end{itemize}
any of the following three statements hold:
\begin{enumerate}[{\rm (A)}]
\item For every large enough regular cardinal $\chi$, 
 every model $M\prec H(\chi)$ with $|M|=\lambda$, $\lambda\subseteq M$ and
${}^{\kappa>}M\subseteq M$, there exists an elementary embedding $j: M
\to N$ with critical point $\kappa$, such that for all $X\subseteq N$
with $|X|\leq\lambda$, there exists $Y\in N$ with $X\subseteq Y$ and
$N \models |Y| < j(\kappa)$.
\item For every large enough regular cardinal $\chi$, for every
$T\in H(\chi)$, there is a model $M\prec H(\chi)$ satisfying $T\in M$,
$|M|=\lambda$, $\lambda\subseteq M$ and ${}^{\kappa>}M\subseteq M$,
with an elementary embedding $j: M \to N$ having critical point
$\kappa$, such that for all $X\subseteq N$ with $|X|\leq\lambda$,
there exists $Y\in N$ with $X\subseteq Y$ and $N \models |Y| <
j(\kappa)$.
\item For every $A\subseteq \lambda$, there is a transitive $M\models
\text{ZFC}^-$ with $\kappa,\lambda \in M$, $A\subseteq M$, $|M|=\lambda$,
${}^{\kappa>}M\subseteq M$, and an $N$ with an elementary
embedding $j: M \to N$ having critical point $\kappa$, such that
$j(\kappa)>\lambda$ and there is an $s\in N$ such that $j''(\lambda)
\subseteq s$ and $|s|^N<j(\kappa)$.
\end{enumerate}
\end{thm}

We prove the theorem through a series of lemmas, assuming throughout
that $\kappa$ is uncountable and $\lambda^{<\kappa}=\lambda$.

\begin{lemma}\label{lemma1}
(A) $\implies$ ($\kappa$ has the
$\lambda$-filter extension property) $\implies$ ($\kappa$ is nearly
$\lambda$-strongly compact) $\implies$ (A).
\end{lemma} 

\begin{proof}[Proof of \emph{(A) $\implies$ ($\kappa$ has the
$\lambda$-filter extension property)}]

Let $\F$ be a $(<\kappa)$-complete filter on $2^\lambda$ which is
generated by a family $\E$ of size $|\E|\leq\lambda$ and let $\A$
be a family of subsets of $2^\lambda$ with $|\A|\leq\lambda$.  Let
$\chi=(2^\lambda)^+$, so $\F, \E, \A \in H(\chi)$. Let $M\prec
H(\chi)$ be a model with $\E,\F,\lambda\in M$, $|M|=\lambda$,
$\lambda\cup\A\cup \E\subseteq M$ and ${}^{\kappa>}M\subseteq
M$. Furthermore, let $j: M \to N$ be an elementary embedding with
critical point $\kappa$ as in (A). For $X:=j''(\F\cap M)$ there exists
$Y\in N$ with $X\subseteq Y$ and $N \models |Y| < j(\kappa)$.

By elementarity of the embedding $j$, $j(\F)$ is a ($<
j(\kappa)$)-complete filter of subsets of $j(2^\lambda)$. Now $j(\F)\cap Y
\subseteq \P(j(2^\lambda))$ and $|j(\F)\cap Y| < j(\kappa)$. Consequently, $N
\models \bigcap (j(\F)\cap Y) \neq\emptyset$, and therefore there
exists some $c\in N\cap \big(\bigcap (j(\F)\cap Y)\big)$. We shall use
$c$ to define a filter $\F^*$ on $2^\lambda$ with $\F\subseteq \F^*$ and such
that all sets in $\A$ are measured by $\F^*$.

To define $\F^*$ we first define an $M$-ultrafilter $\F'$, that is a
family $\F'\subseteq \P(2^\lambda)$ such that $\F'$ is a filter and it
measures every element of $\P(2^\lambda)\cap M$. For $A\in
\P(2^\lambda)\cap M$, we let
\[
A\in \F' \quad \iff \quad c\in j(A)\,.
\]
\begin{itemize}
\item If $A\in \F\cap M$ then $j(A) \in j''(\F\cap M) = X\subseteq
Y$. Moreover, it follows from elementarity of $j$, that $j(A) \in
j(\F)$ since $A\in \F$. Thus, $j(A) \in j(\F)\cap Y$ and therefore,
$c\in j(A)$. We have shown that $\F\cap M\subseteq \F'$. Since
$\E\subseteq M\cap \F$ and $\E$ generates $\F$, it follows that
in fact $\F\subseteq \F'$.

\item Let us show that $\F'$ is ($<\kappa$)-complete: Given a family
  $\{A_i: i< i^* < \kappa\}\subseteq \F'$, we have
\[
j\left(\bigcap_{i < i^*}A_i\right) = \bigcap_{i < i^*}j(A_i) \ni c\,.
\]
Consequently, $\bigcap_{i < i^*}A_i \in \F'$. Note that $\{A_i:i <
  i^*\} \in M$ since ${}^{\kappa>}M\subseteq M$.

\item We next show that $\F'$ is an $M$-ultrafilter on $2^\lambda$. Indeed, if
  $A\in \P(2^\lambda)\cap M$, then $c\in j(A)$ or $c\in j(2^\lambda)\setminus j(A) =
  j(2^\lambda\setminus A)$.
\end{itemize}

Since $\A\subseteq \P(2^\lambda)\cap M$, we have that $\F'$ measures all of
$\A$. We then use $\F'$ to generate (in $V$) a ($<\kappa$)-complete filter
$\F^*$ on $2^\lambda$ by letting for $A\subseteq 2^\lambda$,
\[
A\in \F^*\quad \iff \quad A\supseteq B \text{ for some } B\in \F'\,.
\]
It is standard to verify that $\F^*$ is as required.
\end{proof}

\begin{proof}[Proof of \emph{($\kappa$ has the $\lambda$-filter
extension property) $\implies$ ($\kappa$ is nearly $\lambda$-strongly
compact)}]

Note that by our assumptions $|\P_\kappa(\lambda)|=\lambda$, so using
any fixed bijection between $\P_\kappa(\lambda)$ and $\lambda$, we
can treat filters on $\P_\kappa(\lambda)$ as if they were filters on
$\lambda$. Note also that the $\lambda$-filter extension property of
$\kappa$ implies that every $(<\kappa)$-complete filter $\F_0$ on
$\lambda$ generated by a family $\E$ consisting of $\le\lambda$ sets,
can be extended to a $(<\kappa)$-complete filter on $\lambda$ which
measures a given family $\A$ of $\leq\lambda$ subsets of
$\lambda$. Namely, we shall consider $\lambda$ and $2^\lambda$ as
ordinals where $\lambda<2^\lambda$. Then we can define
$\F_0^+=\{B\subseteq 2^\lambda:\, B\cap \lambda\in \F_0\}$. This is
clearly a $(<\kappa)$-complete filter on $2^\lambda$ and is generated
by $\E$. By the $\lambda$-filter extension property of $\kappa$, there is a
$(<\kappa)$-complete filter $\F^+$ on $2^\lambda$ which extends
$\F_0^+$ and measures $\A$. Let $\F=\{B\cap \lambda:\, B\in \F^+\}$. Clearly
$\F_0\subseteq \F$. Since $\lambda \in \F_0$, we have $\lambda\in
\F_0^+$ and so $\lambda\in \F^+$. This suffices to prove that the
$(<\kappa)$-completeness of $\F^+$ implies the
$(<\kappa)$-completeness of $\F$. Finally, it is clear that $\F$
measures $\A$.

Consider a family $\A$ of $\lambda$-many subsets of
$\P_\kappa(\lambda)$ and let $B_\alpha = \{a\in
\P_\kappa(\lambda):\,\alpha \in a\}\subseteq \P_\kappa(\lambda)$ for
every $\alpha<\lambda$. Let $\F_0$ be the filter on
$\P_\kappa(\lambda)$ generated by the family $\B =
\{B_\alpha:\alpha<\lambda\}$, which is of size at most $\lambda$. Then
$\F_0$ is ($<\kappa$)-complete. By the above remarks, the
$\lambda$-filter extension property of $\kappa$ implies that there is
a $(<\kappa)$-complete filter $\F$ on $\P_\kappa(\lambda)$ with $\F_0
\subseteq \F$ that measures $\A$. The filter $\F$ is fine because each
$B_\alpha$ is already in $\F_0$. Finally, it is non-principal, since
if $\bigcap \F\neq \emptyset$, then $\bigcap \F_0\neq \emptyset$ and
in particular there is $a\in \P_\kappa(\lambda)$ with $a\in
\bigcap_{\alpha<\lambda} B_\alpha$. Hence $\lambda\subseteq a$, which
is impossible since $a\in \P_\kappa(\lambda)$ and $\lambda\ge\kappa$.
\end{proof}

\begin{proof}[Proof of \emph{($\kappa$ is nearly
$\lambda$-strongly compact) $\implies$ (A)}] 

Let $\chi$ and $M$ be as in (A). Note that such an $M$ exists since
$\lambda^{<\kappa}=\lambda$.  Let $S = \P_{\kappa}(\lambda)$. Since
$\lambda\subseteq M$ and ${}^{\kappa>}M\subseteq M$, we have
$\P_{\kappa}(\lambda)\subseteq M$. For $a\in S$, let $C_a = \{b\in
\P_{\kappa}(\lambda) :\,a\subseteq b\}$ and $\E = \{C_a:\,a\in
\P_{\kappa}(\lambda)\}$. Next, let $\F$ be the non-principal
($<\kappa$)-complete fine filter on $S$ measuring every element of
$\P(S)\cap M$, which exists by our assumption. Note that $\F\supseteq
\E$ and is an $M$-ultrafilter on $S$.

We now form a subset of the ultrapower $M^S/\F$, where we consider
only those functions $f:S\to M$ that are in $M$, which we denote by
$(M^S/\F)_M$. Since $\F$ decides every element of $\P(S)\cap M$,
$(M^S/\F)_M$ is an \emph{$M$-ultrapower}, that is the relations $f
=_{\F}g$ and $f \in_{\F}g$ are still equivalence relations, as the
sets $\{a:f(a) = g(a)\}$ and $\{a:f(a) \in g(a)\}$ are both in $M$ if
$f,g\in M$.

Since $\kappa$ is uncountable and $\F$ is ($<\kappa$)-complete, it is
closed under intersections of length $\leq\omega$. Therefore we one
can prove in a standard way that there cannot be an infinite
decreasing sequence
\[
  [a_0]_\F \ni [a_1]_\F \ni \dots \ni [a_n]_\F \ni \dots
\]
of elements of $(M^S/\F)_M$, see for example \cite[Prop. 5.3,
Pg. 48]{K}. Hence, $(M^S/\F)_M$ is well-founded.

Let $N = (M^S/\F)_M$ and define $j:M\to N$, with
$j(a) = [a]_\F$. It follows that $j$ is an elementary embedding.

Next, let $X =
\{[f_\alpha]_\F: \alpha < \lambda\}$ be a subset of $N$. Define
$F:\P_{\kappa}(\lambda) \to M$, by $F(a) = \{f_\alpha(a):\alpha
\in a\}$, which is $\in M$ since ${}^{\kappa>}M\subseteq
M$. Put $Y = [F]_\F\in N$ and we shall show that $X\subseteq Y$. We need to
show that for every $\alpha < \lambda$ we have $[f_\alpha]_\F \in
[F]_\F$, i.e. that $\{a: f_\alpha(a) \in F(a)\}\in \F$. But
\[
  \{a: f_\alpha(a) \in F(a)\} = \{a:\alpha \in a\}=C_{\{\alpha\}}\in \F_0.
\]
To see
that $N \models |Y| < j(\kappa)$, note that $|F(a)| \leq |a| < \kappa$
for each $a\in S$, so that $N \models |[F]_\F| < j(\kappa)$.

Finally, observe that $j(\alpha)=\alpha$ for $\alpha<\kappa$. Indeed,
if not, let $\alpha < \kappa$ be the least such that $j(\alpha) >
\alpha$. If $j(\alpha)>[f]_\F = \alpha$, then $\{a\in S: f(a)<
\alpha\}\in \F$, and so by ($<\kappa$)-completeness there is a $\beta
< \alpha$ such that $\{a \in S: f(a)=\beta\} \in\F$. But then,
$[f]_\F = j(\beta) = \beta<\alpha$, a contradiction. Taking
$X=\{[\alpha]_\F:\alpha<\kappa\}$ in the argument of the preceeding
paragraph, we see that $\kappa < j(\kappa)$ so that $j$ has critical
point $\kappa$.
\end{proof}

\begin{lemma}\label{lemma2} (A) $\implies$ (B) $\implies$ ($\kappa$
has the $\lambda$-filter extension property).
\end{lemma}

\begin{proof}[Proof of \emph{(A) $\implies$ (B)}] Let us fix any
$M\prec H(\chi)$ satisfying $T\in M$, $|M|=\lambda$, $\lambda\subseteq
M$ and ${}^{\kappa>}M\subseteq M$. This is possible because $T\in
H(\chi)$. We then apply (A) to obtain (B).
\end{proof}

\begin{proof}[Proof of \emph{(B) $\implies$ ($\kappa$ has the
$\lambda$-filter extension property)}] Let $\F_0, \E$ and $\A$ be as in the
definition of the $\lambda$-filter extension property and let $T =
\{\E,\F_0,\A,\lambda, 2^\lambda\}$. Let $\chi=(2^\lambda)^+$. As guaranteed by (B),
we can choose
$M\prec H(\chi)$ satisfying $T\in M$, $|M|=\lambda$, $\lambda\subseteq
M$ and ${}^{\kappa>}M\subseteq M$ and
an embedding $j$ satisfying the properties stated in (B). Since $|\A|=\lambda$ and
$\A,\lambda\in M$, there is a bijection $f:\lambda\to \A$ with $f\in
M$ (by elementarity). Hence, for every $\alpha < \lambda$,
$f(\alpha)\in M$ and in particular, $\A = \{f(\alpha): \alpha <
\lambda\}\subseteq M$. Similarly for $\E$.

Hence, $M$ and $j$ satisfy the assumptions of the proof of (A)
$\implies$ ($\kappa$ has the $\lambda$-filter extension property) above and
the same proof gives us that $\kappa$ has the $\lambda$-filter extension
property holds.
\end{proof}

\begin{proof}[Proof of Theorem \ref{T:EmbCh}] It is proved in \cite{W}, by
an argument credited to Cody and White, that (C) is equivalent to
$\kappa$ being nearly $\lambda$-strongly compact. Together with the
results of Lemma \ref{lemma1} and Lemma \ref{lemma2}, this proves the
theorem.
\end{proof}

\begin{rem} Note that in the proof of (A) $\implies$ ($\kappa$ has the
$\lambda$-filter extension property) in Lemma \ref{lemma1}, the
($<\kappa$)-complete filter $\F$ could have been taken on any set $S$
instead of $2^\lambda$. This shows that Definition \ref{D:2} can be
restated as:
\begin{itemize}
\item For every set $S$ and every family $\A\subseteq \P(S)$ with
$|\A|\leq\lambda$, every ($<\kappa$)-complete filter $\F$ over $S$
generated by at most $\lambda$ many elements can be extended to a
($<\kappa$)-complete filter over $S$ measuring $\A$.
\end{itemize}
\end{rem}

\section{Weakly Compact Cardinals}\label{S:proof}

The following Theorem \ref{T:ECwCC} gives several characterisations of
weakly compact cardinals in terms of productivity of topological
spaces. This theorem, which we believe is of independent interest,
will be used in Section \ref{sec:consistency} to give a lower bound
for the consistency strength of the notion of $\lambda$-square
compactness. Many parts of this theorem are known but not always easy to find in the literature. We in particular point out the direct simple proof of
${\rm (1)}\iff {\rm (3)}$ by Luca Motto Ros in \cite{luca}, while a more general statement and
further references, some of them quite old, can be found in \cite{lucaIJM}.

\begin{thm}\label{T:ECwCC} Suppose that
$\kappa=\kappa^{<\kappa} >\aleph_0$. Then the following are equivalent:
\begin{enumerate}[{\rm (1)}]
\item $\kappa$ is a weakly compact cardinal;
\item Suppose that $\{ X_\gamma: \,\gamma \le \Gamma\}$ for some
$\Gamma\le \kappa$ is a family of $(<\kappa)$-compact spaces such that
$w(X_\gamma) \leq \kappa$ for all $\gamma\in\Gamma$.  Then the
$(<\kappa)$-box product space $\prod_{\gamma \in \Gamma}X_\gamma$ is
also $(<\kappa)$-compact;
\item $2^\kappa$ is $(<\kappa)$-compact in the $(<\kappa)$-box
topology;
\item $\kappa$ has the $\kappa$-filter extension property;
\item Suppose that $\{ X_\gamma: \,\gamma \le \Gamma\}$ for some
$\Gamma\le \kappa$ is a family of $(<\kappa)$-compact spaces such that
$w(X_\gamma) \leq \kappa$ for all $\gamma\in\Gamma$.  Then the
Tychonoff product space $\prod_{\gamma \in \Gamma}X_\gamma$ is also
$(<\kappa)$-compact;
\item $\kappa$ is a weakly square compact cardinal;
\item The product of any two $(<\kappa)$-compact GO-spaces of weight
$\leq\kappa$ is $(<\kappa)$-compact.
\end{enumerate}
\end{thm}

We recall that a {\em generalised ordered space (GO-space)} is a
Hausdorff space equipped with a linear order and having a base of
order-convex sets.

We also recall that a cardinal $\kappa$ is said to have {\em the tree
property} if every tree of cardinality $\kappa$ whose every level has
cardinality $<\kappa$, has a branch of cardinality $\kappa$. It is
well known that $\kappa$ is weakly compact if and only if it is
strongly inaccessible and has the tree property, see \cite{K}.

Many parts of Theorem \ref{T:ECwCC} follow from what has already been
proved. Once we prove the following Lemma \ref{fromtreetotopology}, we
shall be ready for the proof of Theorem \ref{T:ECwCC}.

\begin{lemma}\label{fromtreetotopology} If a cardinal $\kappa$ is
weakly compact then $2^\kappa$ is $(<\kappa)$-compact in the
$(<\kappa)$-box topology.
\end{lemma}

\begin{proof} One can note that $2^\kappa$ can be viewed as a tree $T$
of height $\kappa$ such that each level has cardinality $<\kappa$,
since $\kappa$ is a strong limit. Note that branches are closed sets
in the ($<\kappa$)-box topology described below.

Suppose that $\F$ is a ($<\kappa$)-complete filter of closed sets in
$2^\kappa$ equipped with the $(<\kappa$)-box topology. We define a tree
$\hat{T}$ in the following manner: an element
$s \in 2^\beta$, where $\beta < \kappa$, is in $\hat{T}$ if for every
$F\in\F$ there is a branch $b\in F$ such that $b\upharpoonright \beta
= s$. In other words, the open set $U_s$ intersects $F$ for every
$F\in\F$, where $U_s = \prod_{\alpha<\kappa}V_\alpha$ with $V_\alpha =
\{s(\alpha)\}$ for $\alpha<\beta$ and $V_\alpha = \{0,1\}$ for
$\beta\leq\alpha<\kappa$.

Let us show that for any $\beta < \kappa$, $\hat{T}$ has nodes on the
$\beta^{{\rm th}}$ level. Suppose that for any $s \in 2^\beta$,
$s\notin\hat{T}$. Then for any $s \in 2^\beta$ there exists an
$F_s\in\F$ such that $F_s\cap U_s=\emptyset$. Now
$2^\beta<\kappa$ and $\F$ is ($<\kappa$)-complete so that $F =
\bigcap_{s \in 2^\beta}F_s\in\F$. However, $U = \bigcup_{s \in 2^\beta}U_s =
2^\kappa$ and $F\cap U=\emptyset$, contradicting $F\neq\emptyset$.

Thus, $\hat{T}$ is a $\kappa$-tree. By the tree property, there is a
$\kappa$-branch $b\in [\hat{T}]$ and by definition of $\hat{T}$ it
follows that $b\in F$ for every $F\in\F$. Indeed, if $b\notin F$ for
some $F\in\F$, then there is some $\beta<\kappa$ and an open
neighbourhood $U_b$ of $b$ of the form $\prod_{\alpha<\kappa}V_\alpha$
with $V_\alpha = \{b(\alpha)\}$ for $\alpha<\beta$ and $V_\alpha =
\{0,1\}$ for $\beta\leq\alpha<\kappa$, satisfying $U_b\cap
F=\emptyset$. But $s=b\upharpoonright\beta\in\hat{T}$, so that $U_s =
U_b$ intersects every $F\in\F$, a contradiction. Consequently,
$\bigcap\F\neq\emptyset$.
\end{proof}

We can now give a proof of Theorem \ref{T:ECwCC}.

\begin{proof}[Proof of Theorem \ref{T:ECwCC}] (1) $\implies$ (3) is the content of
Lemma \ref{fromtreetotopology}. Therefore (2), (3) and (4) are equivalent by Theorem
\ref{T:LFEP2}.

Implications (2) $\implies$ (5) $\implies$ (6) $\implies$ (7) are obvious.

We are only left with (7) $\implies$ (1). This follows from the
argument in the paper of Hajnal and Juh\'{a}sz \cite{HJ}, which we
recall as it is not phrased in these terms in the original paper.  By
a result of Hanf in \cite{Hanf}, if $\kappa$ is not weakly compact
then there exists a linearly ordered set $(L,\leq)$ with $|L|\ge
\kappa$, in which every decreasing or increasing well-ordered (by
$\leq$) subset is of cardinality $<\kappa$. By taking a suborder if
necessary, we may assume that $|L|=\kappa$. To get the counterexample
in \cite{HJ}, one has to assume that $(L,\leq)$ is continuously
ordered, that is for all non-empty $A, B\subseteq L$, if $A<B$, then
there is $c$ such that $A<c<B$, where we write $A<B$ for subsets of
$L$ if $a<b$ for all $a\in A$ and $b\in B$, a property which can be
obtained in a standard way by an inductive construction in which all
the gaps $(A, B)$ are filled. This procedure (which leads to what is
called a linear compactification $cL$ of the original $L$) may in
general increase the size of $L$ but in our particular case it does
not, since $\kappa^{<\kappa}=\kappa$.  The linear
compactification will preserve the property that every decreasing or
increasing well-ordered subset is of cardinality $<\kappa$. So we
shall assume that $L$ is already continuous.

To pass to the topological spaces, we take two copies of $L$, which we
denote by $X_r$ and $X_l$. The basic open sets are half-open
intervals of the form $(x,y]$ for $X_r$ and the ones of the form
$[x,y)$ for $X_l$. This makes $X_l$ and $X_r$ GO spaces. Both of these
spaces have size and so weight $\le\kappa$. It is shown in \cite{HJ}
that both $X_l$ and $X_r$ are $(<\kappa)$-compact, while their product
$X_l \times X_r$ is not.
\end{proof}

\section{Consistency strength of $2^\kappa$-square
  compactness}\label{sec:consistency}

As mentioned above, \cite{HJ} proves that a fully square compact
cardinal must be weakly compact and it is clear from the definition
and the equivalences listed in Theorem \ref{T:SCEq} that a strongly
compact cardinal is square compact. We shall consider the consistency
strength of $\lambda$-square compactness, getting the same lower bound
of weak compactness, but a much better upper bound than strong
compactness.

Our main findings can be summarised as follows. We remind the reader
that the definition implies that being $\lambda$-square compact is a
notion that increases in strength with $\lambda$.

\begin{thm}\label{main} Suppose that $\kappa$ is uncountable. Then
\begin{enumerate}[{\rm (1)}]
\item\label{lower} If $\kappa=\kappa^{<\kappa}$ is $\lambda$-square
compact for any $\lambda\ge\kappa$, then $\kappa$ is weakly compact.
\item\label{upper} It is consistent modulo large cardinals that the
first weakly compact cardinal $\kappa$ is $2^\kappa$-square compact.
\end{enumerate}
\end{thm}

Item (\ref{upper}) above implies that from $\kappa$ being
$2^\kappa$-square compact, one cannot infer that $\kappa$ has any
large cardinal strength larger than weak compactness. Namely:

\begin{cor}\label{HJ} Suppose that $\varphi$ is a large cardinal
property such that any $\kappa$ satisfying $\varphi(\kappa)$ has an
unbounded set of weakly compact cardinals below. Then, it is
consistent modulo large cardinals that there is a $\kappa$ which is
$2^\kappa$-square compact, but it does not have property $\varphi$.
\end{cor}

We recall that being strongly compact, being measurable or even just
being indescribable are large notions such that any cardinal
satisfying them has an unbounded set of weakly compact cardinals
below. See \cite{K}.

\begin{proof}[Proof of Theorem \ref{main}] (1) follows from the
Theorem \ref{T:ECwCC}. For item (2), we shall need to refer to the
notion of nearly $\lambda$-supercompactness, defined by Jason Schanker
in \cite{Schanker}. In the case of $\lambda=\lambda^{<\kappa}$,
Theorem 1.4. of \cite{Schanker} gives that $\kappa$ is nearly
$\lambda$-supercompact if and only if for every family $\A$ of $\lambda$-many
subsets of $\P_\kappa(\lambda)$ and a collection $\mathscr G$ of
$\lambda$-many functions from $\P_\kappa(\lambda)$ to $\lambda$, there
is a non-principal ($<\kappa$)-complete fine $\mathscr G$-normal
filter $\F$ over $\P_\kappa(\lambda)$ which measures $\A$.

It is then evident from Definition \ref{D:LNSC} that in the case of
$\lambda=\lambda^{<\kappa}$, every nearly $\lambda$-supercompact
cardinal is nearly $\lambda$-strongly compact and therefore by Theorem
\ref{T:EmbCh}, $\kappa$ has the $\lambda$-filter extension property.

Now let us place ourselves in any universe of set theory in which some
cardinal $\kappa$ is nearly $\lambda$-supercompact for some
$\lambda=\lambda^{<\kappa}$. An example of such a universe is one in
which $\kappa$ is a supercompact cardinal and $\lambda=2^\kappa$,
since any cardinal which is $\lambda$-supercompact is nearly
$\lambda$-supercompact (see Observation 1.2 (1) of
\cite{Schanker}). Theorem 7 of \cite{Codyetal} gives a forcing
extension of such a universe in which $2^\kappa=\lambda$ and $\kappa$
is simultaneously the first weakly compact cardinal and nearly
$\lambda$-supercompact. In particular $\kappa$ is nearly
$2^\kappa$-strongly compact, has the $2^\kappa$-filter extension
property and, by Theorem \ref{T:EmbCh} is $2^\kappa$-square compact.
\end{proof}

It is of course interesting to ask if in Theorem \ref{main} we can say
something about $\lambda$-square compactness for
$\lambda>2^\kappa$. The methods of \cite{Codyetal} cannot help us
here, since they necessarily make $2^\kappa=\lambda$. There is another
way for obtaining near $\lambda$-super compactness without
measurability, as in \cite{Schanker}, but as proved in the same paper,
this method is limited to $\lambda\le\kappa^+$.

We cannot hope to use nearly $\lambda$-strong compactness to show that
the fully square compact cardinals have low consistency strength, as
can be seen by the following theorem.

\begin{thm}\label{wcompact} Suppose that $\kappa$ is nearly
$\lambda$-strongly compact for unboundedly many $\lambda$. Then
$\kappa$ is strongly compact.
\end{thm}
\begin{proof} Indeed, from what is proved above, it follows that in
such a case, $2^{\lambda}$ is $(<\kappa)$-compact in the
$(<\kappa)$-box topology for unboundedly many $\lambda$. Consequently,
any $(<\kappa)$-box product of discrete two-point spaces is
$(<\kappa)$-compact, so that $\kappa$ is strongly compact.
\end{proof}

\ifx\undefined\bysame
\newcommand{\bysame}{\leavevmode\hbox to3em{\hrulefill}\,}
\fi


\begin{thebibliography}{1}

\bibitem{Codyetal} B. Cody, M. Gitik, J.D. Hamkins and J.A. Schanker,
The least weakly compact cardinal can be unfoldable, weakly measurable
and nearly $\theta$-supercompact, {\em Archive for Mathematical
Logic}, {\bf 54}, Issue 5-6 (2015), 491–510.

\bibitem{HJ} A. Hajnal and I. Juh\'{a}sz, \emph{On square-compact
cardinals}, Periodica Mathematica Hungarica {\bf 3}, Issue 3 - 4
(1973), 285--288.

\bibitem{Hanf} W. P. Hanf, \emph{On a problem of Erd{\"o}s and
Tarski}, Fundamenta Mathematicae {\bf 53}, (1964), 325--334.


\bibitem{HS} P. Holy and P. Schlicht, \emph{A hierarchy of Ramsey-like
cardinals}, Fundamenta Mathematicae {\bf 242}, Issue 1 (2018), 49--74.

\bibitem{K} A. Kanamori, {\em The Higher Infinite}, Springer
Monographs in Math., Springer, 2009.

\bibitem{KeislerTarski} H. Keisler and A. Tarski, \emph{From
accessible to inaccessible cardinals (Results holding for all
accessible cardinal numbers and the problem of their extension to
inaccessible ones)}, Fundamenta Mathematicae {\bf 53}, Issue 3 (1964),
225--308.

\bibitem{lucaIJM} P. L{\"u}cke, L. Motto Ros and P. Schlicht,
\emph{The {H}urewicz dichotomy for generalized {B}aire spaces}, Israel Journal of Mathematics,
{\bf 216} (2016), 973--1022.

\bibitem{luca} L. Motto Ros, \emph{The descriptive set-theoretical complexity of the 
embeddability relation on models of large size},
Annals of Pure and Applied Logic, {\bf 164} (2013), 1454-1492.
  
\bibitem{My} J. Mycielski, {\em Two remarks on Tychonoff's product
theorem}, Bull. Acad. Polan. Sci. Ser.  Sci. Math. Astr. Phys. {\bf
12} (1964), 439--441.

\bibitem{Schanker} J.A. Schanker, {\em Partial near supercompactness},
Annals of Pure and Applied Logic, {\bf 164} (2013), pg. 67-85.

\bibitem{Sch} E. Schechter, \emph{Handbook of Analysis and Its
Foundations}, Academic Press, 1996.

\bibitem{W} P. A. White, {\em Some Intuition Behind Large Cardinal
Axioms, Thier Characterization, and Related Results},
M.Sc. Dissertation, VCU Virginia, 2019.

\end{thebibliography}
\end{document}